\documentclass[12pt]{amsart}
\usepackage{amssymb, rotating}
\newtheorem{theorem}{Theorem}[section]
\newtheorem{proposition}[theorem]{Proposition}
\newtheorem{lemma}[theorem]{Lemma}
\newtheorem{corollary}[theorem]{Corollary}

\newtheorem{constr}[theorem]{Construction}

\newtheorem{hypothesis}[theorem]{Hypothesis}

\newtheorem{example}[theorem]{Example}

\newcommand{\GL}{\mathrm{GL}}

\mathchardef\mhyphen="2D

\newcommand{\nl}{\par \bigskip \noindent}

\newcommand{\norml}{\vartriangleleft}

\DeclareMathOperator{\Aut}{Aut}

\DeclareMathOperator{\Sym}{Sym}
\newcommand{\PG}{\mathop{\mathrm{PG}}}

\newcommand{\PSL}{\mathop{\mathrm{PSL}}}

\def\dotcup{\DOTSB\mathop{\overset{\textstyle.}\cup}}
\def\bigdotcup{\DOTSB\mathop{\overset{\textstyle.}\bigcup}}

\title[Basic and degenerate pregeometries]{Basic and degenerate pregeometries}
\author[Michael Giudici, Cai Heng Li, Geoffrey Pearce and Cheryl E.~Praeger]{Michael Giudici, Cai Heng Li, Geoffrey Pearce and Cheryl E.~Praeger \\ \\Centre for the Mathematics of Symmetry and Computation\\ 
School of Mathematics and Statistics\\
University of Western Australia\\
Crawley WA 6009 Australia}
\begin{document}
%\date\today

\begin{abstract}
We study pairs $(\Gamma,G)$, where $\Gamma$ is a `Buekenhout--Tits' pregeometry with all rank 2 truncations connected, and $G\leqslant\Aut\,\Gamma$ is transitive on the set of elements of each type.
The family of such pairs is closed under forming quotients with respect to $G$-invariant type-refining partitions of the element set of $\Gamma$. We identify the `basic' pairs  (those that admit no non-degenerate quotients), and show, by studying quotients and direct decompositions, that the  study of basic pregeometries reduces to examining those where the group $G$ is faithful and primitive on the set of elements of each type. 
We also study the special case of normal quotients, where we take quotients with respect to the orbits of a normal subgroup of $G$. There is a similar reduction for normal-basic pregeometries to those where $G$ is faithful and quasiprimitive on the set of elements of each type. 
%
%In this paper we investigate the problem of understanding pregeometries which are `basic' with respect to taking imprimitive and normal quotients.  We use slightly different definitions of `basic' when considering these two different types of quotients, and thus we have respectively the concepts {\em $G$-primitive-basic} and {\em $G$-normal-basic} (relative to a group $G$ of automorphisms).  We show that the study of $G$-primitive-basic pregeometries reduces to examining those where for each type in the pregeometry, $G$ is faithful and primitive on the set of elements of that type. We also show that there is a similar reduction for $G$-normal-basic pregeometries to those where for each type in the pregeometry, $G$ is faithful and quasiprimitive on the set of elements of that type.
\end{abstract}

\maketitle
\section{Introduction} \label{intro} 

In the 1950s,  Jacques Tits~\cite{T54,T55,T56} gave a unified geometric interpretation for all complex Lie groups, and in particular for the exceptional Lie groups, by means of a new kind of geometry, later called a building. 
According to Tits, `it was really the geometric interpretations of these mysterious groups, the exceptional
groups, that triggered everything'~\cite[page 474]{TT}. In the 1970s, inspired by Tits' work, Francis Buekenhout studied a more general family of incidence geometries aiming to find a geometrical interpretation for all the sporadic finite simple groups which, in conjunction with Tits' geometries, would provide a unified theory for all the (then known) finite simple groups, as well as a `unification of the known classes of combinatorial geometries'. As Buekenhout  explained further (see~\cite{B79}), his approach `reduces combinatorial geometry to graph theory and lattice theory'.

In this paper we analyse a larger family of Buekenhout--Tits geometries that satisfy
mild combinatorial properties (much weaker than buildings) and correspondingly mild transitivity properties (much weaker than flag-transitivity). We seek a rationale that enables us to identify the fundamental or `basic' geometries in the family, and to determine what additional geometric and symmetry properties these basic geometries possess. 
In what might be regarded as a pilot for this research project, we (the first, second and fourth authors) studied the family of connected rank 2, locally 2-transitive  geometries (or in graph theoretic language, locally $s$-arc transitive graphs), identifying and describing the basic examples and their automorphism groups, see \cite{GLP3,GLP4,GLP5,GLP6,GLP7}. 
We wanted to know if a similar combinatorial/geometrical approach would identify the basic geometries in the larger family as some (tractible) class containing the geometries for simple groups as a natural subclass.

The incidence geometries we consider are now called pregeometries, though Tits called them geometries in \cite{T56}. Pregeometries, their rank 2 truncations, and their automorphism groups, are defined in Subsection~\ref{incgeoms}. We study the family  $\mathcal{G}$ of pairs $(\Gamma,G)$, where $\Gamma$ is a pregeometry with all rank 2 truncations connected, and $G\leqslant\Aut\,\Gamma$ is transitive on the set of elements of each type.
The family $\mathcal G$ is closed under taking certain kinds of quotients, and our aim is to understand $\mathcal G$ by first determining the `basic' pairs in $\mathcal G$ (those that admit no non-degenerate quotients) and then describing how all other pairs in $\mathcal G$ are derived from these `basic' ones. 

We consider two kinds of quotient operations on pairs $(\Gamma,G)\in\mathcal G$, namely {\em imprimitive} quotients, where we take quotients with respect to $G$-invariant type-refining partitions of the element set of $\Gamma$, and {\em normal} quotients where we take quotients with respect to the orbits of a normal subgroup of $G$ (the latter is a special case of the former).
We find that it is appropriate to use slightly different definitions of degeneracy according to whether we are considering imprimitive or normal quotients, and this leads to correspondingly different concepts of `basic'. Thus we have respectively the concepts {\em $G$-primitive-basic} and {\em $G$-normal-basic} (defined formally in Subsections \ref{primdefs} and \ref{sec:ndegen} respectively).  

Theorem \ref{primitivebasicthm} shows that the study of $G$-primitive-basic pregeometries reduces to examining those where, for each type in the pregeometry, $G$ is faithful and primitive on the set of elements of that type. Theorem~\ref{normalbasicthm} shows that there is a similar reduction for $G$-normal-basic pregeometries to those where $G$ is faithful and quasiprimitive on the set of elements of each type.  (A permutation group is called {\em quasiprimitive} if every non-trivial normal subgroup is transitive; it is {\em primitive} if the only invariant partitions are trivial.)  
Since a pregeometry can be viewed as a multipartite graph, and the automorphism group of a pregeometry is the stabiliser in the graph automorphism group of each part in the multipartition, these results also provide information about automorphism groups of multipartite graphs.

\smallskip\noindent
{\bf Some notation:}\quad Given two pregeometries $\Gamma_1$ and $\Gamma_2$, the {\em direct sum} $\Gamma_1 \oplus \Gamma_2$ is (loosely speaking) the disjoint union of the two pregeometries with complete incidence between the two element sets (see Section \ref{decomps}).  A pregeometry is called {\em indecomposable} if it is not the direct sum of smaller pregeometries.  We denote by ${\mathcal Q}$ the subset of all pairs $(\Gamma,G)\in \mathcal{G}$ such that $G$ is faithful and quasiprimitive on the set of elements of each type; and we use $\Gamma(k,m)$ to denote the pregeometry whose incidence graph is the complete multipartite graph $K_{k \times m}$ with $k$ parts of size $m$.  These (and other) special terms and notation in the theorems are explained further in Sections \ref{incgeoms} and \ref{decomps} and just before Proposition \ref{basicpregeoms}.

\begin{theorem} \label{primitivebasicthm}
Let $(\Gamma,G)\in\mathcal{G}$ such that $\Gamma$ is $G$-primitive-basic.  Then $\Gamma$ admits a unique decomposition $\Gamma_1 \oplus \ldots \oplus \Gamma_\ell$ where for each $i$, $\Gamma_i$ is indecomposable and the group $G^{\Gamma_i}$ induced on $\Gamma_i$ is faithful and primitive on the set of elements of each type in $\Gamma_i$.
%for each type in $\Gamma_i$, the group $G^{\Gamma_i}$ induced on $\Gamma_i$ is faithful and primitive on the set of elements of that type.
\end{theorem}

\begin{theorem} \label{normalbasicthm}
Let $(\Gamma,G)\in\mathcal{G}$ such that $\Gamma$ is $G$-normal-basic of rank $k$.  Then exactly one of the following holds for some integers $m,m'$.
\begin{itemize}
\item[(i)] $(\Gamma,G) \in {\mathcal Q}$;
\item[(ii)] $\Gamma = \Gamma_0 \oplus \Gamma'$ where $(\Gamma_0,G^{\Gamma_0}) \in {\mathcal Q}$, $G\cong G^{\Gamma_0}$, and either
\begin{itemize}
\item[(a)] $\Gamma_0$ has rank $k - 1\geq 1$, and $\Gamma' = \Gamma(1,m)$, or
\item[(b)] $\Gamma_0$ has rank $k - 2\geq 2$, and $\Gamma' = \Gamma(1,m) \oplus \Gamma(1,m')$;
\end{itemize}
\item[(iii)] $k = 2$ and $\Gamma=\Gamma(1,m)\oplus\Gamma(1,m')$; or $k=3$ and $\Gamma=\Gamma(1,m)\oplus\Gamma(1,m/a)\oplus\Gamma(1,m/b)$ for some $a,b$ dividing $m$; or  $\Gamma = \Gamma(k,m)$ with $3\leq k\leq m+1$;
\item[(iv)] $\Gamma$ has a rank $k-1$ truncation $\Gamma_0$ with $(\Gamma_0,G^{\Gamma_0}) \in {\mathcal Q}$, $G\cong G^{\Gamma_0}$, and $G$ is faithful but not quasiprimitive on the excluded subset of elements.
\end{itemize}
\end{theorem}

We prove a little more for some of these cases in Section \ref{compmultsect}. In particular, if $k\geq4$ in case (iii), then $m$ is a prime power, and the upper bound $k=m+1$ can be achieved for each such $m$, see Example~\ref{abelianpregeom}. We note that, for a $G$-normal quotient of a $G$-incidence transitive pregeometry, incidence in the original pregeometry has a certain uniformity relative to the quotient: for each ordered pair $(i,j)$ of types, there is a constant $k_{ij}$ such that, whenever parts $B_i, B_j$ (consisting of type $i$ and type $j$ elements respectively) are incident in the quotient, each element of $B_i$ is incident with exactly $k_{ij}$ elements of $B_j$ (see \cite[Lemma 6.6]{CGP}). For some applications this uniformity is important, and it does not hold for $G$-imprimitive quotients in general. For such applications only Theorem~\ref{normalbasicthm} can be used.

For a pregeometry to be a geometry each flag must be contained in a chamber (that is, a flag containing an element of each type). If $(\Gamma,G)\in\mathcal G$ with $\Gamma$ a geometry, then in general neither an imprimitive quotient nor a normal quotient of $\Gamma$ can be guaranteed to be a geometry, see \cite{CGP}. Several sufficient conditions were obtained in \cite{CGP} for quotients of geometries to be geometries. 
In a forthcoming paper \cite{secondpaper} we will pursue the study of $G$-primitive basic geometries $\Gamma$, showing that, for each O'Nan--Scott type of primitive group $G$, there are examples of thick geometries $\Gamma$ of unboundedly large rank.

%This has been a very successful approach in graph theory, in particular for characterising distance-transitive graphs \cite{SmithDH}, $s$-arc transitive graphs \cite{PraegerQP} and locally $s$-arc transitive graphs \cite{GLP3}.  

\subsection{Outline of the paper}
Section \ref{prelim} covers preliminary theory of pregeometries including imprimitive and normal quotients, and also some results on permutation group actions.  In Section \ref{primitivebasicproof} we prove Theorem \ref{primitivebasicthm}, and in Section \ref{actions} we give a number of results concerning group actions on normal-basic pregeometries.  We then use these results to prove Proposition \ref{basicpregeoms} which lists the ways in which a group $G$ can act on a normal-basic pregeometry (and which proves most of Theorem \ref{normalbasicthm}).  In Section \ref{compmultsect} we deal explicitly with case (iii) of Proposition \ref{basicpregeoms} (namely when the incidence graph of the pregeometry is complete multipartite), giving detailed information about the automorphism group actions in this case. This then enables us to prove Theorem \ref{normalbasicthm}.

\section{Preliminaries} \label{prelim}
\subsection{Incidence pregeometries} \label{incgeoms}
By a {\em pregeometry} $\Gamma = (X,\ast,t)$ we mean a set $X$ of {\em elements} (often called {\em points}) equipped with an {\em incidence relation} $\ast$ on the points, and a map $t$ from $X$ onto a set $I$ of {\em types}.  The incidence relation is symmetric and reflexive, and if $x \ast y$ we say that $x$ and $y$ are {\em incident}.  Furthermore if $x \ast y$ with $x \neq y$ then $t(x) \neq t(y)$.  For each $i \in I$, we write $X_i$ to denote $t^{-1}(i)$, the set of all elements of type $i$, and so $X$ is the disjoint union $\bigcup_{i \in I} X_i$.  The number of types $|I|$ is called the {\em rank} of the pregeometry.  We assume throughout the paper that $|X|$, and hence $|I|$, is finite.  Unless stated otherwise, the set $I$ for a rank $k$ pregeometry is equal to $\{1,\ldots,k\}$.

A typical example is a projective space: the elements are the non-trivial proper subspaces of a vector space, two subspaces are incident if one is contained in the other, and given a subspace $\alpha$, the type $t(\alpha)$ is the algebraic dimension of $\alpha$.

For a pregeometry $\Gamma = (X,\ast,t)$, and a nonempty subset $J$ of the type set $I$, the $J$-truncation of $\Gamma$ is the pregeometry $(X_J,\ast_J,t_J)$ where $X_J = t^{-1}(J)$, $\ast_J$ is the restriction of $\ast$ to $X_J\times X_J$, and $t_J$ is the restriction of $t$ to $X_J$.

The {\em incidence graph} of $\Gamma$ is the graph with vertex set $X$, and edges $\{x_i,x_j\}$ whenever $x_i \ast x_j$ and $i \neq j$, where $x_i$ has type $i$ and $x_j$ has type $j$.  Since different elements of the same type are never incident, such a graph is always multipartite.  If the incidence graph is complete multipartite (that is, $\{x_i,x_j\}$ is an edge whenever $i \neq j$), then we say that the pregeometry $\Gamma$ is itself complete multipartite.  Furthermore, we use the notation $\Gamma(k,m)$ to mean the rank $k$ complete multipartite pregeometry in which $|X_i| = m$ for all $i$.

Throughout the paper we study pregeometries $\Gamma$ in which the rank $2$ truncations are connected.  This means that for each pair of distinct types $i,j$, the incidence graph of the $\{i,j\}$-truncation of $\Gamma$ is connected.

Let $\Gamma = (X,\ast,t)$ be an incidence pregeometry.  An {\em automorphism} of $\Gamma$ is a permutation $g$ of $X$ such that $t(x) = t(x^g)$ for all $x \in X$, and $x^g \ast y^g$ if and only if $x \ast y$, for all $x, y \in X$.  We write $\Aut\Gamma$ for the automorphism group of $\Gamma$. As in Section \ref{intro}, we write $\mathcal{G}$ for the set of all pairs $(\Gamma,G)$ where $\Gamma$ is a pregeometry with all rank 2 truncations connected, and $G\leqslant\Aut(\Gamma)$ is transitive on each $X_i$.

\subsection{Permutation groups} \label{pergroups}
In this section we introduce some notation and results pertaining to group actions.  Let $G$ be a group acting on a set $Y$.  We write $G_{(Y)}$ for the kernel of the $G$-action on $Y$, and $G^{Y}$ for the group induced by $G$ on $Y$ (isomorphic to $G/G_{(Y)}$).  The action of $G$ on $Y$ is called {\em faithful} if $G_{(Y)}=1$, so that $G^Y \cong G$.  For $y \in Y$, we write $y^G$ for the orbit of $y$ under $G$.  The action of $G$ on $Y$ is called {\em semi-regular} if for all $y \in Y$, the stabiliser $G_y\leqslant G_{(Y)}$.  It is called {\em regular} if it is both semi-regular and transitive.

For a subgroup $T$ of $G$, the {\em centraliser} $C_G(T)$ of $T$ in $G$ is the set of all $g \in G$ which commute with every $t \in T$.  The following result is taken from \cite[Theorem 4.2A]{DM}.

\begin{lemma}[Centraliser Lemma] \label{centlem}
Let $T$ be a transitive subgroup of a permutation group $G$ on $Y$, and let $S \leq C_G(T)$.  Then $S$ is semi-regular on $Y$.  Moreover, if $S$ is transitive on $Y$, then $S$ is regular, $S = C_G(T)$, $T \cong S$, and $T$ is also regular on $Y$.  If in addition $T$ is abelian, then $T = S$. 
\end{lemma}

Our next result is a generalisation of \cite[Lemma 5.4]{GLP3}, which dealt with the case where $k=2$.

\begin{lemma} \label{lem:qp}
Suppose that $G$ acts on a set $X$ with orbits $X_1,X_2,\ldots,X_k$ such that, for all nontrivial normal subgroups $N$ of $G$, $N$ is transitive on all but at most one $G$-orbit. Then $G$ is quasiprimitive on all but at most one $G$-faithful orbit.
\end{lemma}
\begin{proof}
Suppose that there exist distinct $i, j$ such that $G$ is faithful but not quasiprimitive on $X_i$ and $X_j$. Then there exist nontrivial normal subgroups $N_i,N_j$ of $G$ such that $N_i^{X_i}$ and $N_j^{X_j}$ are intransitive and nontrivial. Since each of $N_i$ and $N_j$ is transitive on all but at most one $G$-orbit, it follows that $N_i \neq N_j$ and $N_i^{X_j},N_j^{X_i}$ are transitive. Now  $N_i\cap N_j\norml G$ and is intransitive on both $X_i$ and $X_j$. Hence $N_i\cap N_j=1$  and so $N_j\leqslant C_G(N_i)$ and $N_i\leqslant C_G(N_j)$. By the faithfulness of $G$ on $X_i,X_j$ and by Lemma \ref{centlem}, we have that $N_j$ acts semiregularly and intransitively on $X_j$ and $N_i$ acts semiregularly and intransitively on $X_i$. Thus $|N_j| < |X_j|$ and $|N_i| < |X_i|$. Moreover, by transitivity we have that $|X_j|$ divides $|N_i|$ and $|X_i|$ divides $|N_j|$. Thus $|N_j|< |X_j|\leq |N_i| < |X_i|\leq |N_j|$, which is a contradiction. It follows that $G$ is quasiprimitive on all but at most one $G$-faithful $X_i$.
\end{proof}

\subsection{Quotients of incidence pregeometries}
Let $\Gamma = (X,\ast,t)$ be a pregeometry and let ${\mathcal P}$ be a type-refining partition of $X$; that is, each part of ${\mathcal P}$ is contained in one of the sets $X_i$.  Then the {\em quotient} of $\Gamma$ by ${\mathcal P}$ is the pregeometry $\Gamma_{/\mathcal P} = ({\mathcal P},\ast_{/\mathcal P},t_{/\mathcal P})$ where for $P,P' \in {\mathcal P}$ we have $P \ast_{/\mathcal P} P'$ if and only if there exists $x \in P$ and $y \in P'$ with $x \ast y$; and $t_{/ P}(P) = t(x)$ for $x \in P$.  We say that  $\Gamma_{/\mathcal P}$ is a proper quotient if $|{\mathcal P}| < |X|$.

Let $G \leq \Aut\Gamma$.  A partition $\mathcal{P}$ of $X$ is $G$-invariant if $G$ permutes the parts of $\mathcal{P}$ setwise. If ${\mathcal P}$ is a type-refining $G$-invariant partition of $X$ then $\Gamma_{/\mathcal P}$ is called a $G$-{\em imprimitive} quotient of $\Gamma$, or simply an imprimitive quotient if the group is clear from the context.

Let $N$ be a normal subgroup of $G$, and let ${\mathcal P}_{N}$ be the set of $N$-orbits on $X$.  Then ${\mathcal P}_{N}$ is a $G$-invariant partition of $X$, and the $G$-{\em normal quotient}, or simply normal quotient, of $\Gamma$ by $N$ is the pregeometry $\Gamma_{/{\mathcal P}_N}$.  We usually abbreviate the notation to $\Gamma_{/N}$.

As mentioned in the introduction, the motivation for using quotients to study pregeometries is that it gives a framework for characterising pregeometries as being built up from `basic' examples (that is, examples which admit no `non-degenerate' quotients).  As the terms `degenerate' and `basic' have slightly different meaning in the contexts of imprimitive and normal quotients, we give their definitions separately in Sections \ref{primdefs} and \ref{sec:ndegen}.

\subsection{Decomposition of pregeometries} \label{decomps}
When characterising imprimitive-basic and normal-basic pregeometries we use the notion of indecomposability.  For this we use the direct sum of pregeometries (as it appears in \cite[p 80]{IncGeomHB}).  
\begin{constr} {\em
Let $\Gamma_i = (X^{(i)},\ast^{(i)},t^{(i)})$ with type set $I^{(i)}$ for $i = 1,2$.  The {\em direct sum} of $\Gamma_1$ and $\Gamma_2$ is 
$$\Gamma_1 \oplus \Gamma_2 = (X,\ast,t)$$ where 
$X = X^{(1)} \dotcup X^{(2)}$, $I = I^{(1)} \dotcup I^{(2)}$, $t$ induces $t^{(i)}$ on $X^{(i)}$, $\ast$ induces $\ast^{(i)}$ on the restriction to $\Gamma_i$, and $x^{(1)} \ast x^{(2)}$ for all $x^{(1)} \in X^{(1)}$ and $x^{(2)} \in X^{(2)}$.
}
\end{constr}

Note that $\Aut(\Gamma_1 \oplus \Gamma_2)$ is equal to the direct product $\Aut\Gamma_1 \times \Aut\Gamma_2$.

We say that a pregeometry $\Gamma$ is {\em decomposable} if there exist $\Gamma_1$ and $\Gamma_2$ such that $\Gamma = \Gamma_1 \oplus \Gamma_2$.  If $\Gamma$ is not decomposable then it is called {\em indecomposable}.

\section{Primitive-degenerate and primitive-basic pregeometries}
Throughout this section $\Gamma = (X,\ast,t)$ denotes a pregeometry with type set $I$ and $G \leq \Aut\Gamma$.  For $J \subseteq I$, we say that the $J$-truncation $\Gamma_J$ is {\em fully $G$-faithful}, or {\em fully $G$-primitive} if the group $G^{\Gamma_J}$ induced by $G$ on $\Gamma_J$ is faithful on $X_j$ for each $j \in J$, or primitive on $X_j$ for each $j \in J$, respectively.

\subsection{Definitions} \label{primdefs}
We use the term {\em effective rank} of a pregeometry to mean the number of types which have more than one element.  For our theory of imprimitive quotients we use proper quotients of a pregeometry having the same effective rank as the original pregeometry, or equivalently, no part in the partition should be equal to an entire set $X_i$ in the original pregeometry.

Provided $G$ leaves invariant a proper, non-trivial partition ${\mathcal P}_i$ of at least one of the $X_i$, we can always extend ${\mathcal P}_i$ to a partition ${\mathcal P}$ of $X$ such that ${\mathcal P}$ partitions each of the remaining $X_j$ into singleton subsets (irrespective of the $G$-action on $X_j$), giving a proper quotient pregeometry with the same effective rank.

We say that a pregeometry $\Gamma$ is {\em primitive-degenerate} if at least one of the $X_i$ contains only one element.  The corresponding notion of basic, for a pregeometry $\Gamma$ with $G \leq \Aut\Gamma$, is as follows:  $\Gamma$ is called {\em $G$-primitive-basic} if $\Gamma$ is not primitive-degenerate but every proper $G$-imprimitive quotient is primitive-degenerate.  The latter condition is equivalent to requiring that $\Gamma$ is fully $G$-primitive.

\subsection{Proof of Theorem \ref{primitivebasicthm}} \label{primitivebasicproof}

We begin with the following lemma.
\begin{lemma}
\label{lem:faithpart}
Suppose that $\Gamma=(X,*,t)$ is a pregeometry of rank at least two.  Suppose also that $G\leq \Aut\,\Gamma$, $G$ is transitive on each $X_i$, and $1\neq N\norml G$ such that $N$ acts trivially on $X_1$ and transtively on $X_j$ for some $j\neq 1$. Assume that there exist $x_1 \in X_1$ and $x_j \in X_j$ with $x_1 \ast x_j$.  Then each element of $X_1$ is incident with each element of $X_j$.
\end{lemma}
\begin{proof}
Let $x_1 \in X_1$ and $x_j \in X_j$ such that $x_1 * x_j$. Since $x_1^G=X_1$, $x_1^N=\{ x_1\}$, $x_j^{N} = X_j$ and $G$ preserves incidence, the result follows.
\end{proof}

\begin{lemma}
\label{lem:decompimplyfullyfaith}
Let $(\Gamma, G)\in\mathcal{G}$  and suppose that $\Gamma$ is fully $G$-primitive.  If $\Gamma$ is indecomposable then $\Gamma$ is fully $G$-faithful.
\end{lemma}
\begin{proof}
Assume that $G$ is unfaithful on some set $X_i$.  Let $T_i = G_{(X_i)}$, and let $I_{i}$ be the set of all $j$ with $T_i \leq G_{(X_j)}$.  Let $I' = I \backslash I_{i}$.  Since $G$ is faithful on $X$, $|I'| \geq 1$.  Let $\Gamma_{i}$ be the $I_{i}$-truncation and $\Gamma'$ the $I'$-truncation of $\Gamma$.  Let $j \in I'$ and let $x \in \bigcup_{\ell \in I_i} X_\ell$.  Since the rank $2$ truncations are connected there exists $x_j \in X_j$ with $x \ast x_j$.  Since $G$ is primitive on $X_j$, the non-trivial normal subgroup $T_i$ of $G$ is transitive on $X_j$, while fixing the point $x$ in $\bigcup_{\ell \in I_i} X_\ell$.  Thus by Lemma \ref{lem:faithpart},  $x \ast x_j'$ for all $x_j' \in X_j$.  As this holds for all $x \in \bigcup_{\ell \in I_i} X_\ell$, it follows that $\Gamma = \Gamma_{i} \oplus \Gamma'$, and so $\Gamma$ is decomposable.
\end{proof}

\begin{corollary}
Let $(\Gamma, G)\in\mathcal{G}$ and suppose that $\Gamma$ is fully $G$-primitive.  Then there exists a unique partition ${\mathcal I} = \{I_1,\ldots,I_\ell\}$ of $I$ such that, for each $i$, the truncation $\Gamma_{I_i}$ is an indecomposable, fully $G$-faithful, fully $G$-primitive pregeometry, and $\Gamma = \Gamma_{I_1} \oplus \ldots \oplus \Gamma_{I_\ell}$.
\end{corollary}
\begin{proof}
The existence of such a partition follows from Lemma \ref{lem:decompimplyfullyfaith}, Lemma \ref{lem:faithpart} and the fact that all rank 2 truncations are connected. Suppose that  $\Gamma=\Gamma_{J_1} \oplus \ldots \oplus \Gamma_{J_r}$ is another decomposition and let $n\in\{1,\ldots,\ell\}$.  Now $I_n=(I_n\cap J_1)\dotcup(I_n\cap J_2)\dotcup\ldots\dotcup(I_n\cap J_r)$. For $i=1,2,\ldots,r$, let $\Gamma_i$ be the $(I_n\cap J_i)$-truncation of $\Gamma_{I_n}$. Then $\Gamma_{I_n}=\oplus_{j\in J}\Gamma_{I_j}$, where $J=\{j\,\mid\, 1\leq j\leq r,\ I_n\cap J_j\ne\emptyset\}$. Since $\Gamma_{I_n}$ is indecomposable, it follows that $|J|=1$ so there exists $j\in\{1,\ldots,r\}$ such that $I_n\subseteq J_j$. Applying the same argument to $J_j$ implies that $I_n=J_j$. Hence the partition $\mathcal{I}$ is unique.
\end{proof}

Theorem \ref{primitivebasicthm} follows from the above Corollary.
The following example illustrates a decomposable example in which each of the $\Gamma_i$ is an ``interesting'' geometry.
\begin{example} {\em
Let $\Gamma$ be the geometry whose points are the elements of the projective space $\PG(d-1,q)$ and incidence is the usual incidence (that is, subspace inclusion) with the added incidence that every $1$-space and $2$-space is incident with every subspace of dimension at least $3$.  Let $i$ be the type assigned to the $i$-spaces; so $X_i$ is the set of all $i$-spaces.

Let $G = \PSL(d,q) \times \PSL(d,q)$ act on this geometry where the first factor acts non-trivially on $X_1$ and $X_2$ and trivially on the rest, while the second factor acts trivially on $X_1$ and $X_2$ but non-trivially on the rest.  Since $G$ acts primitively on each type, $\Gamma$ is $G$-primitive-basic.

The group $G$ acts unfaithfully on each of the $X_i$.  However, the incidence graph of $\Gamma$ is not complete multipartite, as $\Gamma = \Gamma_1 \oplus \Gamma_2$ where $\Gamma_1$ and $\Gamma_2$ are, respectively, the $\{1,2\}$-truncation and $\{3,\ldots,d-1\}$-truncation of $\Gamma$, and each of $\Gamma_1$ and $\Gamma_2$ is indecomposable, with each being a truncation of $\PG(d-1,q)$.
}%%
\end{example}

\section{Normal-degenerate and normal-basic pregeometries} \label{actions}

\subsection{Normal degeneracy}\label{sec:ndegen}
Let $(\Gamma,G)\in\mathcal{G}$. 
In the case of normal quotients, we do not have the freedom to choose the induced partitions of the $X_i$ independently of one another, as they are all determined by the orbits of a given normal subgroup $N \unlhd G$.  Thus it is less reasonable (and less useful) to use the notion of primitive-degeneracy in this context.  Instead we require for non-degeneracy only that the quotient pregeometry has effective rank at least two (that is, $|X_i| \geq 2$ for at least two types $i$).

A pregeometry $\Gamma$ is called {\em normal-degenerate} if at most one of the $X_i$ contains more than one element.  Corresponding to this notion, a pregeometry $\Gamma$ with $G \leq \Aut\Gamma$ is called {\em $G$-normal-basic} if $\Gamma$ is not normal-degenerate, and every proper quotient is normal-degenerate.  The latter condition is equivalent to requiring that every non-trivial normal subgroup of $G$ is transitive on all but at most one of the $X_i$.

\subsection{Faithful, unfaithful and quasiprimitive orbits}
\label{sec:fqptypes}

For a permutation group $G$ on a set $X$ we say that a $G$-orbit is {\em $G$-faithful} if $G$ acts faithfully on it, and otherwise it is called {\em $G$-unfaithful}.
Let $\Gamma=(X,*,t)$ be a pregeometry with  $G \leq \Aut\Gamma$.  A type $i \in I$ is called {\em $G$-unfaithful} or {\em $G$-faithful} according to whether $G_{(X_i)} \neq 1$ or $G_{(X_i)} = 1$ respectively, and $i$ is called {\em $G$-quasiprimitive} if $G^{X_i}$ is quasiprimitive.  If every $i \in I$ is $G$-quasiprimitive and $G$-faithful then $\Gamma$ is called {\em fully $G$-quasiprimitive}.

\begin{lemma}
\label{lem:faithful}
Suppose that $(\Gamma,G)\in\mathcal{G}$ has rank at least two and for all nontrivial normal subgroups $N$ of $G$, $N$ is transitive on all but at most one of the $X_i$. Then either  $\Gamma$ is complete multipartite or $G$ is faithful on at least two of the $X_i$.
\end{lemma}
\begin{proof}
We have $X=X_1\cup\ldots\cup X_k$. Let $J$ be the subset of $I$ consisting of all $j$ such that $G$ is unfaithful on $X_j$. If $|I\setminus J|\geq 2$ then the second conclusion follows. Hence assume that $|I\setminus J|\leq 1$ and let $j\in J$. Then $G_{(X_j)} \neq 1$.  Let $i \in I \backslash \{j\}$.  By assumption, $G_{(X_j)}$ is transitive on $X_i$.  By Lemma \ref{lem:faithpart}, every element of type $j$ is incident with every element of type $i$. Since this holds for all $j\in J$ and all $i \neq j$, it follows that the incidence graph of $\Gamma$ is complete multipartite.
\end{proof}

Define $I_{\rm unf}$ to be the set of all $G$-unfaithful types; $I_{\rm qp}$ the set of all types that are both $G$-faithful and $G$-quasiprimitive; and $I_{\rm nonqp}$ the set $I \backslash (I_{\rm unf} \cup I_{\rm qp})$. In particular, we have a partition of $I$ into the three disjoint parts $I_{\rm qp}$, $I_{\rm unf}$ and $I_{\rm nonqp}$.

\begin{lemma}
\label{lem:boundingunfqp} 
Let $(\Gamma,G)\in\mathcal{G}$ such that for all nontrivial normal subgroups $N$ of $G$, $N$ is transitive on all but at most one of the $X_i$. Then $$|I_{\rm unf}\cup I_{\rm nonqp}|\leq \#\text{ minimal normal subgroups of }G.$$
Moreover, if $|I_{\rm qp}|\geq 1$ then $|I_{\rm unf} \cup I_{\rm nonqp}|\leq 2$ with equality implying $|I_{\rm unf}|=2$.
\end{lemma}
\begin{proof}
Let $j\in I_{\rm unf}\cup I_{\rm nonqp}$ and let $m$ be the number of minimal normal subgroups of $G$. Then there exists a minimal normal subgroup $N_j$ of $G$ such that $N_j^{X_j}$ is intransitive. By assumption, for all $i\neq j$,
$N_j^{X_i}$ is transitive. Thus for distinct $j,j'\in I_{\rm unf}\cup I_{\rm nonqp}$. $N_j\neq N_{j'}$ and hence $|I_{\rm unf}\cup I_{\rm nonqp}| \leq m$. Moreover, if $|I_{\rm qp}|\geq 1$ then by \cite[Theorem 4.4]{CameronBook}, $G$ has at most two minimal normal subgroups and so $|I_{\rm unf} \cup I_{\rm nonqp}|\leq 2$.

Suppose now that $|I_{\rm qp}|\geq 1$ and $|I_{\rm unf} \cup I_{\rm nonqp}|=2$. Then $G$ has exactly two minimal normal subgroups $N,M$, say, and in particular, each type in $I_{\rm qp}$ is primitive. Moreover, $N \cong M \cong T^n$ for some finite non-abelian simple group $T$ and positive integer $n$ (see \cite[Theorem 4.3B]{DM}). Also $N \cap M = 1$, so $\langle N,M \rangle=NM$.  Let $i\in I_{\rm unf}\cup I_{\rm nonqp}$. Then without loss of generality we may assume that $N^{X_i}$ is intransitive. Suppose that $j\in  I_{\rm unf}\cup I_{\rm nonqp}$ with $j\neq i$.  By assumption $N^{X_j}$ is transitive and so, for $\alpha\in X_j$,  $G=G_{\alpha}N$.  Note that $N^{X_j}$ being transitive implies that $|X_j|$ divides $|N|$ and that $(NM)^{X_j}$ is transitive.  Let $H=G_{\alpha}\cap (NM)$ and let $\pi_2$ be the projection map from $NM$ to $M$. Then $H\norml G_{\alpha}$ and so $HN$ is normalised by $G_{\alpha}N=G$. Now $HN\leqslant NM$ and $M\cong (NM)/N$ is a minimal normal subgroup of $G/N$. Hence $HN/N=MN/N$ and so $HN=MN$. Thus $M\cong H/(H\cap N)$ and so $\pi_2(H)=M$. 
Now $H\cap M=G_{\alpha}\cap (NM)\cap M=G_{\alpha}\cap M$, which is normalised by $G_{\alpha}$. Since $G=G_{\alpha}N$ and $N$ centralises $M$, $G_{\alpha}$ acts transitively on the set of simple direct factors of $M$. Moreover, as $H\cap M\norml H$ we have $\pi_2(H\cap M)\norml \pi_2(H)=M$. Thus $H\cap M=M$ or $1$. Since $(NM)^{X_j}$ is transitive, we have $|X_j| = |NM:H|$.  If $H \cap M = 1$ then $M_\alpha=1$ and  $|NM:H|=|NM:HM||HM:H|=|NM:HM||M|$ and so $|M|$ divides $|NM:H| = |X_j|$. As noted above, $|X_j|$ divides $|N| = |M|$, so $|M:M_\alpha|=|M| = |X_j|$ and $M^{X_j}$ is transitive. Thus each minimal normal subgroup of $G$ is transitive on $X_j$ and hence $G$ is quasiprimitive on $X_j$, which is a contradiction. Thus $H\cap M=M$, and in particular $M$, a normal subgroup of $G$ is contained in $G_{\alpha}$. Hence $G$ acts unfaithfully on $X_j$ and so $j\in I_{\rm unf}$. Moreover, applying the same argument with $i$ and $j$ interchanged we find that $N$ acts trivially on $X_i$, and $i\in I_{\rm unf}$. Thus $|I_{\rm unf} \cup I_{\rm nonqp}| = 2$, which implies $I_{\rm nonqp}=\varnothing$. 
\end{proof}

\subsection{Group actions on normal-basic pregeometries}
Let $\Gamma = (X,\ast,t)$ be a pregeometry and $G \leq \Aut\Gamma$. Recall that if $G$ has a faithful quasiprimitive action on each $X_i$ then we say that $\Gamma$ is {\em fully $G$-quasiprimitive}.  Let ${\mathcal Q}$ denote the set of all pairs $(\Gamma,G)$ where $\Gamma$ is a pregeometry and $G \leq \Aut\Gamma$ such that $\Gamma$ is fully $G$-quasiprimitive.  Recall also that we use the notation $\Gamma(k,m)$ to denote the rank $k$ complete multipartite pregeometry in which $|X_i| = m$ for all $i$.

Our next Proposition lists all the possible ways a group $G$ can act on a normal-basic pregeometry.

\begin{proposition} \label{basicpregeoms}
Let $(\Gamma,G)\in\mathcal{G}$ such that $\Gamma$ is a $G$-normal basic pregeometry of rank $k$. Then exactly one of the following holds.
\begin{itemize}
\item[(i)] $(\Gamma,G) \in {\mathcal Q}$.
\item[(ii)] 
\begin{itemize}
\item[(a)] $k \geq 3$ and $\Gamma = \Gamma_0 \oplus \Gamma(1,m)$ where $(\Gamma_0,G^{\Gamma_0}) \in {\mathcal Q}$, $G\cong G^{\Gamma_0}$, and the single type of $\Gamma(1,m)$ is $G$-unfaithful.
\item[(b)] $k \geq 4$ and $\Gamma = \Gamma_0 \oplus \Gamma(1,m) \oplus \Gamma(1,m')$ where $(\Gamma_0,G^{\Gamma_0}) \in {\mathcal Q}$, $G\cong G^{\Gamma_0}$, and $G$ is unfaithful on the point set of each of $\Gamma(1,m)$ and $\Gamma(1,m')$.  In this case $G$ has $2$ minimal normal subgroups.
\end{itemize}
\item[(iii)] $\Gamma$ is complete multipartite and $G$ is faithful on at most one of the $X_i$.
\item[(iv)] $G$ is faithful but not quasiprimitive on one of the $X_i$, and faithful and quasiprimitive on each of the others.
\end{itemize}
\end{proposition}
\begin{proof} Let $s = |I_{\rm unf} \cup I_{\rm nonqp}|$ (as defined before Lemma \ref{lem:boundingunfqp}).  If $s = 0$ then $\Gamma$ is fully quasiprimitive as in case (i) of the statement.  Suppose next that $s = 1$.  If $|I_{\rm nonqp}| = s=1$ then we are in case (iv).  
If $|I_{\rm unf}| = s=1$ and $k = 2$ then we are in case (iii) by Lemma \ref{lem:faithpart}. Assume next that $|I_{\rm unf}| = s=1$ and $k \geq 3$. Let $I_{\rm unf}=\{j\}$ and let $\Gamma_0$ be the $(I \backslash \{j\})$-truncation and $\Gamma'$ the $\{j\}$-truncation of $\Gamma$.  Then $\Gamma' = \Gamma(1,m)$ where $m = |X_j|$.  Since $\Gamma$ is $G$-normal-basic, so is $\Gamma_0$, and as $G$ is faithful and quasiprimitive on all $X_i$ in $\Gamma$ except $X_j$, $\Gamma_0$ is fully $G$-quasiprimitive.
That $\Gamma = \Gamma_0 \oplus \Gamma(1,m)$ is an immediate consequence of Lemma \ref{lem:faithpart}.  Hence we are in case (ii)(a).

Suppose now that $s = 2$.  Then by Lemma \ref{lem:boundingunfqp}, $G$ has (at least) $2$ minimal normal subgroups.  Moreover, by Lemma \ref{lem:qp}, $|I_{\rm nonqp}| \leq 1$.  Thus $|I_{\rm unf}| \geq 1$.  If $k = 2$ then we are in case (iii) by Lemma \ref{lem:faithpart}.  On the other hand if $k \geq 3$ then $|I_{\rm qp}|=k-s\geq 1$ and so $G$ has at most two minimal normal subgroups (see \cite[Theorem 1]{PraegerQP}).  Since $s=2$, Lemma \ref{lem:boundingunfqp} implies that there are exactly two minimal normal subgroups.  Also, Lemma \ref{lem:boundingunfqp} implies that $|I_{\rm unf}| = 2$.  By Lemma \ref{lem:faithful}, either $k =3$ and we are in case (iii), or $k \geq 4$.  Assume that $k \geq 4$, let $I_{\rm unf}=\{j,\ell\}$, let $\Gamma_0$ be the $(I \backslash \{j,\ell\})$-truncation and $\Gamma'$ the $\{j,\ell\}$-truncation of $\Gamma$. Let $m = |X_j|$ and $m' = |X_\ell|$. By assumption $G_{(X_j)}$ is transitive on $X_{\ell}$, so by Lemma \ref{lem:faithpart}, every point in $X_j$ is incident with every point in $X_\ell$, and hence $\Gamma' = \Gamma(1,m) \oplus \Gamma(1,m')$.  Since $\Gamma$ is $G$-normal-basic, so is $\Gamma_0$, and as $G$ is faithful and quasiprimitive on each $X_i$ with $i\in I\backslash\{j,\ell\}$, $\Gamma_0$ is fully $G$-quasiprimitive.  That $\Gamma = \Gamma_0 \oplus \Gamma'$ also follows from Lemma \ref{lem:faithpart} (applied with $X_1 = X_j$ and then again with $X_1 = X_\ell$ in its statement).  Hence we are in case (ii)(b).
 
Suppose finally that $s > 2$.  Then by Lemma \ref{lem:boundingunfqp}, $|I_{\rm qp}|=\varnothing$ and so $s=k$. Moreover, Lemma \ref{lem:qp} implies that $|I_{\rm unf}|\geq k-1$ and so by Lemma \ref{lem:faithful}, $\Gamma$ is complete multipartite.  Hence we are in case (iii) and the proof is complete.
\end{proof}

\section{Case (iii) of Proposition \ref{basicpregeoms}} \label{compmultsect}
In Proposition \ref{compmultprop} we give further details (shown in Table \ref{compmulttable}) about $G$-normal-basic pregeometries in case (iii) of Proposition \ref{basicpregeoms}. We work with the following hypothesis concerning $\Gamma$, $G$ and $k$.

\begin{hypothesis} \label{compmulthypo}
Suppose that $\Gamma = (X, \ast, t)$ is a $G$-normal-basic pregeometry of rank $k \geq 2$, such that $\Gamma$ is in case (iii) of Proposition {\em \ref{basicpregeoms}}, that is, $\Gamma$ is complete multipartite, and $G$ is faithful on at most one of the $X_i$.  For each type $i$, let $T_i = G_{(X_i)}$.
\end{hypothesis}

First we show via construction that this case arises for rank 2 with $|I_{unf}|=2$ and any $|X_1|, |X_2|$.

\begin{constr} \label{rank2bipart} {\em
Let $X_1$ and $X_2$ be sets and for each $i$ let $T_i$ be a group acting faithfully and quasiprimitively on $X_i$.  Let $\Gamma$ be the pregeometry with point set $X_1 \dotcup X_2$ such that the incidence graph is complete bipartite.  Let $G = T_1 \times T_2$ and let $G$ act on $\Gamma$ by
$$x_i^{(t_1,t_2)} = x_i^{t_i}$$
if $x_i \in X_i$. Clearly $G\leqslant\Aut\Gamma$ and $\Gamma$ is $G$-normal basic. Furthermore, $G$ acts unfaithfully on each $X_i$.
}
\end{constr}

Our main result for $k\geq 3$ is Proposition \ref{compmultprop}.  We use the notation $I_{\rm unf}$ and $I_{\rm qp}$ as defined in Section \ref{sec:fqptypes}.

\begin{proposition} \label{compmultprop}
Suppose that $\Gamma$ and $G$ satisfy Hypothesis {\em \ref{compmulthypo}} with $k\geq 3$.  Then each $T_i$ is a minimal normal subgroup, and $k$, $|I_{\rm unf}|$, $|I_{\rm qp}|$ and the $T_i$ are as in one of the lines of Table {\em \ref{compmulttable}}.  Moreover, if  $|I_{\rm unf}| \geq 3$ then $\Gamma = \Gamma(k,m)$ for some $m$.
\end{proposition}

{\openup 5pt
\begin{table}[h]
\begin{tabular}{|ccccl|} 
\hline
& $k$ & $|I_{\rm unf}|$ & $|I_{\rm qp}|$ & Comments on $T_i$\\
\hline\hline
(i) & $3$ & $2$ & $1$ & nonabelian, $T_1\cong T_2$, $|X_3|=|T_1|$, $|X_1|$ and $|X_2|$ divide $|X_3|$\\
(ii) & $3$ & $3$ & $0$ & nonabelian, $\langle T_1, T_2, T_3 \rangle  = T_1 \times T_2 \times T_3$\\
(iii) & $3 \leq k \leq m+1$ & $k$ & $0$ & $T_i = \mathbb{Z}_p^d$, $m = p^d$, $p$ a prime, $\langle T_i \, | \, i \in I \rangle = \mathbb{Z}_p^{2d}$\\
\hline
\end{tabular} \caption{The possibilites in Proposition \ref{compmultprop}.}  \label{compmulttable}
\end{table}
}

We prove the above Proposition using Lemmas \ref{regularkernels} and \ref{propcasea}--\ref{propcaseb}, and Corollary \ref{samesize} below.

\begin{lemma} \label{regularkernels}
Suppose that $\Gamma$ and $G$ satisfy Hypothesis {\em \ref{compmulthypo}} with $k\geq 3$. For distinct $i,j\in I_{\rm unf}$ and $\ell\in I\backslash\{i,j\}$, the following hold. 
\begin{itemize}
\item[(a)] $T_i \cap T_j = 1$,
\item[(b)] $T_i$ and $T_j$ are faithful and regular on $X_\ell$, and
\item[(c)] $T_i \cong T_j$.
\end{itemize}
\end{lemma}
\begin{proof}
For any distinct $i',j'\in I$, the subgroups $T_{i'}$ and $T_{j'}$ are normal in $G$, so $T_{i'} \cap T_{j'} \lhd G$ and $T_{i'} \cap T_{j'}$ acts trivially on both $X_{i'}$ and $X_{j'}$.  Since $\Gamma$ is $G$-basic, this implies that $T_{i'} \cap T_{j'} =1$. In particular part (a) is proved, and $T_i,T_j$ are both faithful on $X_\ell$.  

Now, since $\Gamma$ is $G$-basic, $T_i^{X_\ell}$ and $T_j^{X_\ell}$ are both transitive.  Since $T_i \cap T_j$ is trivial and $T_i$ and $T_j$ normalise each other, we have $T_i \leq C_G(T_j)$.  Hence by Lemma \ref{centlem}, $T_i$ and $T_j$ are regular on $X_\ell$, and $T_i \cong T_j$, whence we get parts (b) and (c).
\end{proof}

\begin{corollary} \label{samesize}
Suppose that $\Gamma$ and $G$ satisfy Hypothesis {\em \ref{compmulthypo}}, and $|I_{\rm unf}|\geq 3$. Then  $\Gamma = \Gamma(k,m)$ for some $m$.
\end{corollary}
\begin{proof}
Let $i,j\in I_{\rm unf}$ with $i\neq j$, let $\ell\in I\backslash\{i,j\}$, and let $i'\in I_{\rm unf}\backslash\{i,j\}$.  By Lemma \ref{regularkernels}, $T_i$ is regular and faithful on $X_\ell$. The same argument with $\{i,i'\}$, $j$ in place of $\{i,j\}$, $\ell$ gives that $T_i$ is regular and faithful on $X_j$.  So $|T_i| = |X_j| = |X_\ell|$.  Similarly $|X_i| = |X_\ell|$.  Writing $m := |X_\ell|$ we obtain the result.
\end{proof}

\begin{lemma} \label{propcasea}
Suppose that $\Gamma$ and $G$ satisfy Hypothesis {\em \ref{compmulthypo}} such that $k = 3$, and $|I_{\rm unf}| = 2$, say $I_{\rm unf}=\{1,2\}$.  Then $G$ is quasiprimitive on $X_3$ and $T_1, T_2$ are its minimal normal subgroups. Moreover, $\Gamma=\Gamma(1,m_1)\oplus\Gamma(1,m_2)\oplus\Gamma(1,m_3)$ where $m_3=|T_1|=|T_2|$ and $m_1,m_2$ divide $m_3$.
\end{lemma}
\begin{proof}
The fact that $\Gamma$ is $G$-basic implies that $T_i^{X_3}\cong T_i$ is transitive for $i = 1,2$.  By Lemma \ref{regularkernels}, $T_1\cap T_2=1$ so $T_1\leqslant C_G(T_2)$ and hence by Lemma \ref{centlem}, $T_1\cong T_2$ and both are faithful and regular on $X_3$. Hence $|T_1|=|T_2|=|X_3|=m_3$. Moreover, $T_1^{X_2}\cong T_1$, $T_2^{X_1}\cong T_2$ are both transitive, and so, for $i\in\{1,2\}$, $m_i=|X_i|$ divides $m_3=|T_{3-i}|$. Thus, since $\Gamma$ is complete multipartite by Hypothesis~\ref{compmulthypo}, $\Gamma$ is as stated. Also since $G\cong G^{X_3}$ and $T_1\cap T_2=1$, it follows by Lemma \ref{centlem} that $T_1=C_G(T_2)$.

Now assume that $N\unlhd G$ with $N$ intransitive on $X_3$.  Then for $i = 1,2$, $N \cap T_i \lhd G$ and $N \cap T_i$ is intransitive on both $X_i$ and $X_3$, so since $\Gamma$ is $G$-basic, $N \cap T_i$ is trivial. Hence $N\leqslant C_G(T_i)=T_{3-i}$, so $N\leqslant T_1\cap T_2=1$. Thus all nontrivial normal subgroups of $G$ are transitive on $X_3$, that is, $G^{X_3}$ is quasiprimitive. Since $T_1^{X_3}$ and $T_2^{X_3}$ are regular, it follows that $T_1, T_2$ are the minimal normal sugroups of $G$.
\end{proof}

For isomorphic groups $A, B$ a full diagonal subgroup of $A\times B$ is a subgroup $C\leq A\times B$ such that $C$ projects onto both $A$ and $B$, and for each $a\in A$ there is a unique $b\in B$ such that $(a,b)\in C$. Each full diagonal subgroup is of the form $C=\{(a,a^\varphi)\,|\,a\in A\}$ for some isomorphism $\varphi: A\rightarrow B$. 

\begin{lemma} \label{propcasec}
Suppose that $\Gamma$ and $G$ satisfy Hypothesis {\em \ref{compmulthypo}} with $k\geq 3$ and $|I_{\rm unf}|\geq 3$, and that $T_i = G_{(X_i)}$ is abelian for some type $i$.  Then for some prime $p$ and integer $d$, $\Gamma=\Gamma(k,p^d)$, where $k\leq p^d+1$,  each $T_i\cong \mathbb{Z}_p^d$, is a minimal normal subgroup of $G$, and $\langle T_i \, | \, i \in I \rangle = \mathbb{Z}_p^{2d}$.
\end{lemma}
\begin{proof}  First we note that, by Lemma \ref{regularkernels} (c), $T_i \cong T_u$ for all types $u$, so $T_u$ is abelian also.

Let $i$ and $j$ be distinct types, and let $N$ be a minimal normal subgroup of $G$ contained in $T_i$.  Now $T_i$ is abelian, so $N = \mathbb{Z}_p^d$ for some prime $p$.  Moreover, by Lemma \ref{regularkernels} (b), $T_i$ acts faithfully and regularly on $X_j$.  Thus $N$ is faithful and semi-regular on $X_j$ and so as $\Gamma$ is $G$-basic, $N = T_i$ so $T_i$ is a minimal normal subgroup of $G$, and $|N| = |X_j|$.  Since there are at least $3$ unfaithful types, repeating this argument we have $|X_\ell| = p^d$ for all $\ell$.  It follows from the fact that $\Gamma$ is complete multipartite that $\Gamma = \Gamma(k,m)$ where $m = p^d$.

Again, let $i$ and $j$ be distinct types, and note that $T_i \cong T_j \cong \mathbb{Z}_p^d$ and $T_i \times T_j \leq G$ by Lemma \ref{regularkernels}.  

\nl
{\em Claim}: For each type $\ell \neq i,j$, $T_\ell$ is a full diagonal subgroup of $T_i \times T_j$ and $\langle T_r \, | \, r \in I \rangle=T_i\times T_j\cong \mathbb{Z}_p^{2d}$.

\nl
{\em Proof of claim}:  Since $T_i$ and $T_j$ centralise each other, the groups $T_i^{X_\ell}$ and $T_j^{X_\ell}$ centralise each other as subgroups of $\Sym(X_\ell)$, and since each of these subgroups is abelian Lemma \ref{centlem} implies that $T_i^{X_\ell} = T_j^{X_\ell}$.

Thus, let $S_\ell := \{t_it_j \, | \, t_i \in T_i, t_j \in T_j, t_j^{X_\ell} = (t_i^{X_\ell})^{-1}\}$.  
Since $[T_i,T_j]=1$ the set $S_{\ell}$ is a subgroup of $T_i\times T_j$. In fact it is a full diagonal subgroup of $T_i \times T_j \leq G$, implying that $S_\ell \cong T_i \cong T_j$, and also $S_\ell$ acts trivially on $X_\ell$.  It follows from Lemma \ref{regularkernels} (c) that $S_\ell = T_\ell$.  Hence the claim is proved.

\smallskip
Without loss of generality we let $\{i,j\}=\{1,2\}$. Identifying each of $T_1$ and $T_2$ with the vector space $V' := V(d,p)$ we may identify the group induced by the conjugation action of $G$ on $T_1 \times T_2$ with a subgroup $\hat{G}$ of $\GL(2d,p)$ in its action on $V := V' \oplus V'$.  Under this identification, the group $T_\ell$ corresponds to a subspace $V_A := \{(u,uA) \, | \, u \in V'\}$ for some matrix $A$ in $\GL(d,p)$.

Let $g \in G$.  As $G$ normalises each of $T_1$ and $T_2$, there exist $B_1$ and $B_2$ in $\GL(d,p)$ such that the element $\hat{g}$ of $\hat{G}$ induced by $g$ is equal to
$$\left( \begin{array}{c|c} B_1 & {\bf 0}\\
\hline
 {\bf 0} & B_2
\end{array} \right)$$
where ${\bf 0}$ denotes the $d \times d$ zero-matrix.  Since $T_\ell \unlhd G$, we have $V_A^{\hat{g}} = \{(uB_1,uAB_2) \, | \, u \in V'\} = V_A$, and this implies that $AB_2 = B_1A$, and hence $B_2 = A^{-1}B_1A$.  Let $H \leq \GL(d,p)$ be the image of $\hat{G}$ under the projection
$$\left( \begin{array}{c|c} B_1 & {\bf 0}\\
\hline
 {\bf 0} & B_2
\end{array} \right) \longmapsto B_1.$$

Now, let $s\in I\backslash\{1,2,\ell\}$. By the Claim, $T_s$ corresponds to a subspace $V_{A'} = \{(u,uA') \, | \, u \in V'\}$ for some $A' \in \GL(d,p)$.  Arguing as above we find that $B_2 = (A')^{-1}B_1A'$, and hence that $A^{-1}B_1A = (A')^{-1}B_1A'$.  Thus, for all $B \in H$, $B = A'A^{-1}BA(A')^{-1}$ and so $A'A^{-1}$ centralises $H$.  As $T_1$ is a minimal normal subgroup of $G$, it follows that $H$ acts irreducibly on $V'$.  By Schur's Lemma (see \cite[(12,4)]{A}), the centraliser of $H$ in the matrix algebra $M(d,p)$ is  a finite division ring and hence $C_{\GL(d,p)}(H)\cong \GL(1,p^e)$ for some $e$ dividing $d$. Thus $|I\backslash\{1,2\}|\leq p^e-1\leq p^d-1$.
\end{proof}

We give an example to show that the bound $k=p^d+1$ can be achieved.  

\begin{example} \label{abelianpregeom}
{\em
Let $T = \mathbb{Z}_p^d$ for some prime $p$ and integer $d$.  Identify $T$ with the additive group of the finite field $\mathbb{F}_{p^d}$ of order $p^d$, and note that $\mathbb{F}_{p^d}^*$ acts transitively by field multiplication on $T^*$.  Let $G = (T \times T).H$ where $H = \{(h,h) \, | \, h \in \mathbb{F}_{p^d}^*\}$.  Then $T \times 1$ and $1 \times T$ are minimal normal subgroups of $G$.
For all $\lambda \in \mathbb{F}_{p^d}$ define $T_\lambda = \{(u,\lambda u)  \, | \, u \in T\}$.  Also, define $T_{\infty} = 1 \times T$.
Let $\{0,\infty\} \subseteq \Lambda \subseteq \mathbb{F}_{p^d}$.  For each $\lambda \in \Lambda$, define $X_\lambda = [G: T_\lambda.H]$, the set of right cosets of $T_\lambda.H$ in $G$, with $G$ acting on $X_\lambda$ by right multiplication.  Thus $G_{(X_\lambda)} = T_\lambda$.

Define the pregeometry $\Gamma_\Lambda = (X,\ast,t)$ by $X = \bigdotcup_{\lambda \in \Lambda} X_\lambda$, $t(x) = \lambda$ whenever $x \in X_\lambda$, and $x \ast y$ if either $x = y$ or $t(x) \neq t(y)$.
It is straightforward to verify that $\Gamma_{\Lambda} \cong \Gamma(k,m)$ where $m = p^d$ and $k = |\Lambda| \leq p^d +1$ (and if $\Lambda = \mathbb{F}_p^d \cup \{\infty\}$ then $k = p^d + 1$), and that $\Gamma, G$ satisfy Hypothesis~\ref{compmulthypo}.
}
\end{example}

\begin{lemma} \label{compmultabelian}
Suppose that $\Gamma$ and $G$ satisfy Hypothesis {\em \ref{compmulthypo}} with $k \geq 4$.  Let $i$ be an unfaithful type.  Then $T_i$ is abelian.
\end{lemma}
\begin{proof}
By Hypothesis \ref{compmulthypo}, since $k \geq 4$, $\Gamma$ has at least $3$ unfaithful types, say $i,j,\ell$.  Let $r \in I \backslash \{i,j,\ell\}$.  By Lemma \ref{regularkernels}, $T_i^{X_r},T_j^{X_r}$ and $T_{\ell}^{X_r}$ are all regular and $T_i^{X_r}, T_j^{X_r} \leq C_{\Sym(X_r)}(T_\ell^{X_r})$. Lemma \ref{centlem} implies that $T_i^{X_r} = C_{\Sym(X_r)}(T_\ell^{X_r})= T_j^{X_r}$. Moreover, by Lemma \ref{regularkernels}, $T_i^{X_r}$ and $T_j^{X_r}$ centralise each other, so $T_i^{X_r}$ is abelian.  By Lemma \ref{regularkernels}(b), $T_i^{X_r} \cong T_i$, and the result follows.
\end{proof}

\begin{lemma} \label{propcaseb}
Suppose that $\Gamma$ and $G$ satisfy Hypothesis {\em \ref{compmulthypo}}, and assume also that $k = 3$, all types are $G$-unfaithful, and for some type $s$, $T_s$ is non-abelian.  Then the $T_i$ are isomorphic nonabelian minimal normal subgroups of $G$, $\Gamma=\Gamma(3,m)$ with $m=|T_1|$, and $\langle T_1, T_2, T_3 \rangle = T_1 \times T_2 \times T_3$, as in line (ii) of Table {\em \ref{compmulttable}}.
\end{lemma}
\begin{proof}
By Lemma \ref{regularkernels} (a), $\langle T_i, T_j \rangle = T_i \times T_j$ for all $i,j$.  Also, by Lemma \ref{regularkernels}(c) (and since $T_s$ is nonabelian), $T_i\cong T_s$ is nonabelian for all $i$.    By Lemma \ref{regularkernels}(b), for all $i\neq j$, $T_i$ is regular and faithful on $X_j$. It follows that $\Gamma=\Gamma(3,m)$ with $m=|T_1|$. Let $N$ be a minimal normal subgroup of $G$ contained in $T_i$. Since $\Gamma$ is $G$-normal-basic for $j\neq i$, $N^{X_j}$ is transitive and as $T_i^{X_j}$ is faithful and regular on $X_j$ it follows that $N=T_i$. Thus $T_i$ is a minimal normal subgroup of $G$.

Let $1,2,3$ be the three types and let $\hat{T_3} = T_3 \cap (T_1 \times T_2)$.  If $\hat{T_3}\neq 1$ then 
as $\Gamma$ is $G$-normal-basic, $\hat{T_3}$ is transitive on $X_1$ and $X_2$ and hence $|T_3|=|X_1|$ 
divides $|\hat{T_3}|$. Thus $T_3=\hat{T_3}\leqslant T_1\times T_2$. Since $T_1$ and $T_2$ are direct 
products of nonabelian simple groups and are minimal normal subgroups of $G$, the only minimal normal 
subgroups of $G$ contained in $T_1\times T_2$ are $T_1$ and $T_2$, a contradiction. Hence $\hat{T_3}=1$ and 
so $\langle T_1, T_2, T_3 \rangle = T_1 \times T_2 \times T_3$.
\end{proof}

We are now ready to prove Proposition \ref{compmultprop}.

\begin{proof}[Proof of Proposition \ref{compmultprop}.]  By Hypothesis  \ref{compmulthypo}, $|I_{\rm unf}|\geq k-1\geq 2$. Assume first that $|I_{\rm unf}| = 2$.  Then since by assumption $k \geq 3$ and $\Gamma$ has at most one faithful type, we have $k = 3$, and Lemma \ref{propcasea} shows that $|I_{\rm qp}| = 1$, as in line (i) of Table \ref{compmulttable}.
 
Now suppose that $|I_{\rm unf}| \geq 3$.  Then Corollary \ref{samesize} implies that $\Gamma = \Gamma(k,m)$ for some $m$.  Assume first that $k = 3$, and let $i$ be a type.  If $T_i = G_{(X_i)}$ is non-abelian then Lemma \ref{propcaseb} gives line (ii) of Table \ref{compmulttable}, and if $T_i$ is abelian then Lemma \ref{propcasec} gives line (iii).  Finally, assume that $k \geq 4$.  Then Lemma \ref{compmultabelian} shows that the conditions of Lemma \ref{propcasec} hold, giving line (iii) again.
\end{proof}

\begin{proof}[Proof of Theorem \ref{normalbasicthm}] By Proposition \ref{basicpregeoms}, either (i), (ii) or (iv) holds or $\Gamma$ is complete multipartite and $G$ is faithful on at most one of the $X_i$, that is, Hypothesis \ref{compmulthypo} holds. If $k=2$ then $\Gamma=\Gamma(1,m)\oplus\Gamma(1,m')$ by Lemma~\ref{lem:faithpart}, while if $k\geq 3$ then Proposition \ref{compmultprop} yields the two remaining cases of the theorem.
\end{proof}

\section{Acknowledgements}
The paper forms part of the ARC Discovery Project Grant DP0770915.
The authors acknowledge support from the Australian Research Council: an Australian Research Fellowship (first author) and Federation Fellowship (fourth author).

%\bibliography{geomrefs}
%\bibliographystyle{plain}

\end{document}